\documentclass[11pt]{article}
\usepackage{fullpage}
\usepackage{amsfonts, amsmath, amssymb}

\newcommand{\vecx}{{\vec X}}
\newcommand{\veca}{{\vec a}}
\newcommand{\vecb}{{\vec b}}
\newcommand{\veci}{{\vec i}}
\newcommand{\mult}{{\mathrm {mult}}}
\newcommand{\A}{{\mathsf {A}}}
\newcommand{\B}{{\mathsf {B}}}
\newcommand{\C}{{\mathsf {C}}}
\newcommand{\X}{{\mathsf {X}}}
\newcommand{\Y}{{\mathsf {Y}}}
\newcommand{\ZZ}{{\mathsf {Z}}}
\newcommand{\U}{{\mathsf {U}}}

\newcommand{\DW}{{\mathrm {DW}}}
\newcommand{\GUV}{{\mathrm {GUV}}}
\newcommand{\TS}{{\mathrm {TS}}}
\newcommand{\RSW}{{\mathrm {RSW}}}
\newcommand{\HIGH}{{\mathrm {HIGH}}}

\begin{document}

\hbadness=10000
\vbadness=10000

\newtheorem{theorem}{Theorem}
\newtheorem{corollary}[theorem]{Corollary}
\newtheorem{lemma}[theorem]{Lemma}
\newtheorem{observation}[theorem]{Observation}
\newtheorem{proposition}[theorem]{Proposition}
\newtheorem{definition}[theorem]{Definition}
\newtheorem{claim}[theorem]{Claim}
\newtheorem{fact}[theorem]{Fact}
\newtheorem{assumption}[theorem]{Assumption}

\newcommand{\qed}{\rule{7pt}{7pt}}
\newcommand{\dis}{\mathop{\mbox{\rm d}}\nolimits}
\newcommand{\per}{\mathop{\mbox{\rm per}}\nolimits}
\newcommand{\area}{\mathop{\mbox{\rm area}}\nolimits}
\newcommand{\cw}{\mathop{\rm cw}\nolimits}
\newcommand{\ccw}{\mathop{\rm ccw}\nolimits}
\newcommand{\DIST}{\mathop{\mbox{\rm DIST}}\nolimits}
\newcommand{\OP}{\mathop{\mbox{\it OP}}\nolimits}
\newcommand{\OPprime}{\mathop{\mbox{\it OP}^{\,\prime}}\nolimits}
\newcommand{\ihat}{\hat{\imath}}
\newcommand{\jhat}{\hat{\jmath}}
\newcommand{\abs}[1]{\mathify{\left| #1 \right|}}

\newenvironment{proof}{\noindent{\bf Proof}\hspace*{1em}}{\qed\bigskip}
\newenvironment{proof-sketch}{\noindent{\bf Sketch of Proof}\hspace*{1em}}{\qed\bigskip}
\newenvironment{proof-idea}{\noindent{\bf Proof Idea}\hspace*{1em}}{\qed\bigskip}
\newenvironment{proof-of-lemma}[1]{\noindent{\bf Proof of Lemma #1}\hspace*{1em}}{\qed\bigskip}
\newenvironment{proof-attempt}{\noindent{\bf Proof Attempt}\hspace*{1em}}{\qed\bigskip}
\newenvironment{proofof}[1]{\noindent{\bf Proof
of #1:}}{\qed\bigskip}
\newenvironment{remark}{\noindent{\bf Remark}\hspace*{1em}}{\bigskip}


\newcommand{\FOR}{{\bf for}}
\newcommand{\TO}{{\bf to}}
\newcommand{\DO}{{\bf do}}
\newcommand{\WHILE}{{\bf while}}
\newcommand{\AND}{{\bf and}}
\newcommand{\IF}{{\bf if}}
\newcommand{\THEN}{{\bf then}}
\newcommand{\ELSE}{{\bf else}}

\makeatletter
\def\fnum@figure{{\bf Figure \thefigure}}
\def\fnum@table{{\bf Table \thetable}}
\long\def\@mycaption#1[#2]#3{\addcontentsline{\csname
  ext@#1\endcsname}{#1}{\protect\numberline{\csname
  the#1\endcsname}{\ignorespaces #2}}\par
  \begingroup
    \@parboxrestore
    \small
    \@makecaption{\csname fnum@#1\endcsname}{\ignorespaces #3}\par
  \endgroup}
\def\mycaption{\refstepcounter\@captype \@dblarg{\@mycaption\@captype}}
\makeatother

\newcommand{\figcaption}[1]{\mycaption[]{#1}}
\newcommand{\tabcaption}[1]{\mycaption[]{#1}}
\newcommand{\head}[1]{\chapter[Lecture \##1]{}}
\newcommand{\mathify}[1]{\ifmmode{#1}\else\mbox{$#1$}\fi}
\newcommand{\bigO}O
\newcommand{\set}[1]{\mathify{\left\{ #1 \right\}}}
\def\half{\frac{1}{2}}
\def\polylog{{\mathrm{polylog}}}


\newcommand{\enc}{{\sf Enc}}
\newcommand{\dec}{{\sf Dec}}
\newcommand{\E}{{\rm Exp}}
\newcommand{\Var}{{\rm Var}}
\newcommand{\Z}{{\mathbb Z}}
\newcommand{\F}{{\mathbb F}}
\newcommand{\K}{{\mathbb K}}
\newcommand{\integers}{{\mathbb Z}^{\geq 0}}
\newcommand{\R}{{\mathbb R}}
\newcommand{\Q}{{\cal Q}}
\newcommand{\eqdef}{{\stackrel{\rm def}{=}}}
\newcommand{\from}{{\leftarrow}}
\newcommand{\vol}{{\rm Vol}}
\newcommand{\poly}{{\rm {poly}}}
\newcommand{\ip}[1]{{\langle #1 \rangle}}
\newcommand{\wt}{{\rm {wt}}}
\renewcommand{\vec}[1]{{\mathbf #1}}
\newcommand{\mspan}{{\rm span}}
\newcommand{\rs}{{\rm RS}}
\newcommand{\RM}{{\rm RM}}
\newcommand{\Had}{{\rm Had}}
\newcommand{\calc}{{\cal C}}
\def\prob{{\mathbf{Pr}}}

\def\omm{ \{0,1\} }
\def\eps{ \epsilon }

\newcommand{\fig}[4]{
        \begin{figure}
        \setlength{\epsfysize}{#2}
        \vspace{3mm}
        \centerline{\epsfbox{#4}}
        \caption{#3} \label{#1}
        \end{figure}
        }

\newcommand{\ord}{{\rm ord}}
\def\blfootnote{\xdef\@thefnmark{}\@footnotetext}
\providecommand{\norm}[1]{\lVert #1 \rVert}
\newcommand{\embed}{{\rm Embed}}
\newcommand{\qembed}{\mbox{$q$-Embed}}
\newcommand{\calh}{{\cal H}}
\newcommand{\lp}{{\rm LP}}

\setlength{\parindent}{0in}
\setlength{\parskip}{\medskipamount}

\title{Extensions to the Method of Multiplicities, with applications
to Kakeya Sets and Mergers}

\author{%
Zeev Dvir\thanks{IAS. {\tt zeev.dvir@gmail.com}. Research
partially supported by NSF Grant CCF-0832797 (Expeditions in
computing grant) and by NSF Grant DMS-0835373 (pseudorandomness
grant).} \and Swastik Kopparty
\thanks{MIT CSAIL. {\tt swastik@mit.edu}.
Research supported in part by NSF Award CCF 0829672.}
\and Shubhangi Saraf \thanks{MIT
CSAIL. {\tt shibs@mit.edu}.
Research supported in part by NSF Award CCF 0829672.}
\and Madhu Sudan\thanks{MIT CSAIL. {\tt
madhu@mit.edu}.
Research supported in part by NSF Award CCF 0829672.}%
}

\maketitle



\begin{abstract}
We extend the ``method of multiplicities''
to get the following
results, of interest in combinatorics and randomness extraction.
\begin{enumerate}
\item
We show that every Kakeya set (a set of points that
contains a line in every direction) in $\F_q^n$
must be of size at least $q^n/2^n$.
This bound is tight to within a $2 + o(1)$ factor for
every $n$ as $q \to \infty$, compared to previous
bounds that were off by exponential factors in $n$.
\item
We give improved randomness extractors and ``randomness mergers''.
Mergers are seeded functions that take as input $\Lambda$ (possibly
correlated) random variables in $\{0,1\}^N$ and a short random
seed and output a single random variable in $\{0,1\}^N$ that is
statistically close to having entropy $(1-\delta) \cdot N$ when
one of the $\Lambda$ input variables is distributed uniformly. The
seed we require is only $(1/\delta)\cdot \log \Lambda$-bits long,
which significantly improves upon previous construction of
mergers. \item Using our new mergers, we show how to construct
randomness extractors that use logarithmic length seeds while
extracting $1 - o(1)$ fraction of the min-entropy of the source.
Previous results could extract only a constant fraction of the
entropy while maintaining logarithmic seed length.
\end{enumerate}
The ``method of multiplicities", as used in prior work,
analyzed subsets of vector spaces over
finite fields by constructing somewhat low degree
interpolating
polynomials that vanish on every point in the subset {\em
with high multiplicity}.
The typical use of this method involved showing that the
interpolating polynomial also vanished on some points outside
the subset, and then used simple bounds on the number
of zeroes to complete the analysis. Our augmentation to this technique
is that we prove, under appropriate conditions,
that the interpolating polynomial vanishes {\em with high multiplicity}
outside the set. This novelty leads to significantly tighter analyses.

To develop the extended method of multiplicities we provide a number
of basic technical results
about multiplicity of zeroes of polynomials that may be of general
use. For instance, we strengthen the
Schwartz-Zippel lemma to show that the expected multiplicity
of zeroes of a non-zero degree $d$ polynomial at a random point in
$S^n$, for any finite subset $S$ of the underlying field, is at most
$d/|S|$ (a fact that does not seem to have been noticed in the
CS literature before).
\end{abstract}

\newpage

\section{Introduction}
\label{sec:intro}

The goal of this paper is to improve on an
algebraic method that has lately been applied, quite
effectively, to analyze
combinatorial parameters of subsets of vector spaces that
satisfy some given algebraic/geometric conditions.
This technique, which we refer to as
as the {\em polynomial method} (of combinatorics),
proceeds in three steps: Given the subset $K$ satisfying
the algebraic conditions, one first constructs a non-zero low-degree
polynomial that
vanishes on $K$. Next, one uses the algebraic
conditions on $K$ to show that the polynomial vanishes at other
points outside $K$ as well. Finally, one uses the fact that the
polynomial
is zero too often to derive bounds on the combinatorial
parameters of interest.
The polynomial method has seen utility in
the computer science literature in works on ``list-decoding''
starting with Sudan \cite{Sudan} and subsequent works.
Recently the method has been applied to analyze ``extractors''
by Guruswami, Umans, and Vadhan~\cite{GUV}.
Most relevant to this current paper are its applications to
lower bound the cardinality of
``Kakeya sets'' by Dvir~\cite{Dvir08}, and the subsequent constructions
of ``mergers'' and ``extractors'' by Dvir and Wigderson~\cite{DW08}.
(We will elaborate on some of these results shortly.)

The {\em method of multiplicities}, as we term it, may be
considered an extension of this method. In this extension
one constructs polynomials that vanish with {\em high multiplicity}
on the subset $K$. This requirement often forces one to use
polynomials of higher degree than in the polynomial method, but
it gains in the second step by using the high multiplicity of
zeroes to conclude ``more easily'' that the polynomial is zero
at other points. This typically leads to a tighter analysis
of the combinatorial parameters of interest. This method has
been applied widely in list-decoding starting with the work of
Guruswami and Sudan~\cite{GuSu} and continuing through many subsequent
works, most significantly in the works of
Parvaresh and Vardy~\cite{PV} and
Guruswami and Rudra~\cite{GuRu} leading to rate-optimal list-decodable
codes.
Very recently this method was also applied to improve the lower bounds
on the size of ``Kakeya sets'' by Saraf and Sudan~\cite{SaSu}.

The main contribution of this paper is an extension to this
method, that we call the {\em extended method of multiplicities},
which develops this method (hopefully) fully to derive even
tighter bounds on the combinatorial parameters.
In our extension, we start as in the method of multiplicities
to construct a polynomial that vanishes with high multiplicity
on every point of $K$. But then we extend the second step where
we exploit the algebraic conditions to show that the polynomial
vanishes with {\em high multiplicity} on some points outside
$K$ as well. Finally we extend the third step to show that this
gives better bounds on the combinatorial parameters of interest.

By these extensions we derive nearly optimal lower bounds on the
size of Kakeya sets and qualitatively improved analysis of mergers
leading to new extractor constructions.
We also rederive algebraically a known bound on the list-size
in the list-decoding of Reed-Solomon codes. We describe these
contributions in detail next, before going on to describe
some of the technical observations used to derive the extended
method of multiplicities (which we believe are of independent
interest).

\subsection{Kakeya Sets over Finite Fields}

Let $\F_q$ denote the finite field of cardinality $q$.
A set $K \subseteq \F_q^n$ is said to be a {\em Kakeya
set} if it ``contains a line in every direction''.
In other words, for every ``direction'' $\vec b \in \F_q^n$ there
should exist an ``offset'' $\vec a \in \F_q^n$ such that
the ``line'' through $\vec a$ in direction $\vec b$, i.e.,
the set $\{\vec a + t \vec b | t \in \F_q\}$, is contained
in $K$. A question of interest in combinatorics/algebra/geometry,
posed originally by Wolff~\cite{Wolff},
is: ``What is the size of the smallest Kakeya set, for a
given choice of $q$ and $n$?''

The trivial upper bound on the size of a Kakeya set is $q^n$ and
this can be improved to roughly $\frac{1}{2^{n-1}}q^n$ (precisely
the bound is $\frac{1}{2^{n-1}}q^n + O(q^{n-1})$, see \cite{SaSu}
for a proof of this bound due to Dvir). An almost trivial lower
bound is $q^{n/2}$ (every Kakeya set ``contains'' at least $q^n$
lines, but there are at most $|K|^2$ lines that intersect $K$ at
least twice). Till recently even the exponent of $q$ was not known
precisely (see \cite{Dvir08} for details of work prior to 2008).
This changed with the result of \cite{Dvir08} (combined with an
observation of Alon and Tao) who showed that for every $n$, $|K|
\geq c_n q^n$, for some constant $c_n$ depending only on $n$.

Subsequently the work~\cite{SaSu} explored the growth of the constant
$c_n$ as a function of $n$. The result of \cite{Dvir08} shows that
$c_n \geq 1/n!$, and \cite{SaSu} improve this bound to show
that $c_n \geq 1/(2.6)^n$. This still leaves a gap between the
upper bound and the lower bound and we effectively close this gap.

\begin{theorem}
\label{thm:kakeya-crude}
If $K$ is a Kakeya set in $\F_q^n$ then
$|K| \geq \frac{1}{2^n}q^n.$
\end{theorem}

Note that our bound is tight to within a $2+o(1)$ multiplicative factor as
long as $q = \omega(2^n)$ and in particular when $n = O(1)$ and $q \to
\infty$.

\subsection{Randomness Mergers and Extractors}

A general quest in the computational study of randomness is
the search for simple primitives that manipulate random variables
to convert their randomness into more useful forms. The exact notion of
utility varies with applications.
The most common notion is that of
``extractors'' that produce an output variable
that is distributed statistically close to uniformly on the range.
Other notions of interest include ``condensers'', ``dispersers'' etc.
One such object of study (partly because it is useful to construct
extractors) is a ``randomness merger''. A randomness merger
takes as input $\Lambda$, possibly correlated, random variables
$\mathsf A_1,\ldots, \mathsf A_\Lambda$, along
with a short uniformly random seed $\mathsf B$, which is independent of
$\mathsf A_1,\ldots, \mathsf A_\Lambda$, and ``merges'' the randomness of $\mathsf A_1,\ldots, \mathsf A_\Lambda$.
Specifically the output of the merger should be statistically close
to a high-entropy-rate source of randomness provided at least
one of the input variables $\mathsf A_1,\ldots, \mathsf A_\Lambda$ is uniform.

Mergers were first introduced by Ta-Shma \cite{TaShma96} in the
context of explicit constructions of extractors. A general
framework was given in \cite{TaShma96} that reduces the problem of
constructing good extractors into that of constructing good
mergers. Subsequently, in \cite{LRVW03}, mergers were used in a
more complicated manner to create extractors which were optimal to
within constant factors. The mergers of \cite{LRVW03} had a very
simple algebraic structure: the output of the merger was a random
linear combination of the blocks over a finite vector space. The
\cite{LRVW03} merger analysis was improved in \cite{DS07} using
the connection to the finite field Kakeya problem and the (then)
state of the art results on Kakeya sets.

The new technique in \cite{Dvir08} inspired  Dvir and Wigderson
\cite{DW08} to give a very simple, algebraic, construction of a
merger which can be viewed as a derandomized version of the
\cite{LRVW03} merger. They associate the domain of each random
variable $\mathsf A_i$ with a vector space $\F_q^n$. With the
$\Lambda$-tuple of random variables $\mathsf A_1,\ldots, \mathsf A_\Lambda$,
they associate a curve $C:\F_q \to \F_q^n$ of degree $\leq \Lambda$
which `passes' through all the points $\mathsf A_1,\ldots, \mathsf
A_\Lambda$ (that is, the image of $C$ contains these points). They then
select a random point $u \in \F_q$ and output $C(u)$ as the
``merged" output. They show that if $q \geq \poly(\Lambda \cdot n)$ then
the output of the merger is statistically close to a distribution
of entropy-rate arbitrarily close to $1$ on $\F_q^n$.

While the polynomial (or at least linear) dependence of $q$ on $\Lambda$
is essential to the construction above, the requirement $q \geq
\poly(n)$ appears only in the analysis. In our work we remove this
restriction to show:

{\bf Informal Theorem [Merger]:} {\em For every $\Lambda,q$ the
output of the Dvir-Wigderson merger is close to a source of
entropy rate $1 - \log_q \Lambda$. In particular there exists an
explicit merger for $\Lambda$ sources (of arbitrary length) that
outputs a source with entropy rate $1 - \delta$ and has seed
length $(1/\delta)\cdot \log(\Lambda/\eps)$ for any error $\eps$.}

The above theorem (in its more formal form given in
Theorem~\ref{thm:themerger}) allows us to merge $\Lambda$ sources using
seed length which is only logarithmic in the number of sources and
does not depend entirely on the length of each source. Earlier
constructions of mergers required the seed to depend either
linearly on the number of blocks \cite{LRVW03, Zuc07} or to depend
also on the {\em length} of each block \cite{DW08}. \footnote{The
result we refer to in \cite[Theorem 5.1]{Zuc07} is actually a condenser (which
is stronger than a merger).}

One consequence of our improved merger construction is an improved
construction of extractors. Recall that a $(k,\epsilon)$-extractor
$E:\{0,1\}^n \times \{0,1\}^d \to \{0,1\}^m$ is a deterministic
function that takes any random variable $\X$ with min-entropy at
least $k$ over $\{0,1\}^n$ and an independent uniformly distributed
seed $\Y \in \{0,1\}^d$ and converts it to the random variable
$E(\X,\Y)$ that is $\epsilon$-close in statistical distance to a
uniformly distributed random variable over $\{0,1\}^m$.
Such an extractor is efficient if $E$ is polynomial time computable.

A diverse collection of efficient extractors are known in the
literature (see the survey \cite{Sha02} and the more recent
\cite{GUV,DW08} for references) and many applications have been
found for explicit extractor is various research areas spanning
theoretical computer science. Yet all previous constructions lost
a linear fraction of the min-entropy of the source (i.e., acheived
$m = (1 -\epsilon)k$ for some constant $\epsilon > 0$) or used
super-logarithmic seed length ($d = \omega(\log n)$). We show that
our merger construction yields, by combining with several of the
prior tools in the arsenal of extractor constructions, an
extractor which extracts a $1 - \frac{1}{\mathrm{polylog(n)}}$
fraction of the minentropy of the source, while still using
$O(\log n)$-length seeds. We now state our extractor result in an
informal way (see Theorem~\ref{thm-final-extractor} for the formal
statement).

{\bf Informal Theorem [Extractor]:} {\em There exists an explicit
$(k, \epsilon)$-extractor for all min-entropies $k$ with $O(\log n)$ seed,
entropy loss $O(k/ \polylog(n))$ and error $ \eps =
1/\polylog(n)$, where the powers in the $\polylog(n)$ can be
arbitrarily high constants.}

\subsection{List-Decoding of Reed-Solomon Codes}

The Reed-Solomon list-decoding problem is the following: Given a
sequence of points $$(\alpha_1,\beta_1),\ldots,(\alpha_n,\beta_n)
\in \F_q\times \F_q,$$ and parameters $k$ and $t$, find the list
of all polynomials $p_1,\ldots,p_{L}$ of degree at most $k$ that
agree with the given set of points on $t$ locations, i.e., for
every $j \in \{1,\ldots,L\}$ the set $\{ i | p_j(\alpha_i) =
\beta_i\}$ has at least $t$ elements. The associated combinatorial
problem is: How large can the list size, $L$, be for a given
choice of $k,t,n,q$ (when maximized over all possible set of
distinct input points)?

A somewhat nonstandard, yet reasonable, interpretation of the
list-decoding algorithms of \cite{Sudan,GuSu} is that they give
algebraic proofs, by the polynomial method and the method of
multiplicities, of known combinatorial upper bounds on the list
size, when $t > \sqrt{kn}$. Their proofs {\em happen} also to
be algorithmic and so lead to algorithms to find a list of
all such polynomials.

However, the bound given on the list size in the above works does not
match the best known combinatorial bound. The best known bound
to date seems to be that of Cassuto and Bruck~\cite{CaBr} who show that,
letting $R= k/n$ and $\gamma = t/n$, if  $\gamma^2 > R$, then the list size $L$ is bounded by
$O( \frac{\gamma}{\gamma^2-R})$ (in contrast, the Johnson bound
and the analysis of \cite{GuSu} gives a list size bound of $O(\frac{1}{\gamma^2 - R})$,
which is asymptotically worse for, say, $\gamma = (1+ O(1))\sqrt{R}$ and $R$ tending to $0$).
In Theorem~\ref{thm:rsbound} we recover the bound of~\cite{CaBr} using
our extended method of multiplicities.


\subsection{Technique: Extended method of multiplicities}

The common insight to all the above improvements is that
the extended method of multiplicities can be applied to
each problem to improve the parameters. Here we attempt
to describe the technical novelties in the
development of the extended method of multiplicities.

For concreteness, let us take the case of the Kakeya set
problem. Given a set $K \subseteq \F_q^n$, the method first
finds a non-zero polynomial $P \in \F_q[X_1,\ldots,X_n]$ that vanishes
with high multiplicity $m$ on each point of $K$. The next step is
to prove that $P$ vanishes with fairly high multiplicity $\ell$ at
every point in $\F_q^n$ as well. This step turns out to be somewhat
subtle (and is evidenced by the fact that the exact relationship
between $m$ and $\ell$ is not simple).
Our analysis here crucially uses the fact that the (Hasse)
derivatives of the polynomial $P$, which are the central to
the notion of multiplicity of roots,
are themselves polynomials,
and also vanish with high multiplicity at points in $K$.
This fact does not seem to have been
needed/used in prior works and is central to ours.

A second important technical novelty arises in the final
step of the method of multiplicities, where
we need to conclude that if the degree of $P$ is ``small'',
then $P$ must be identically zero. Unfortunately in our application
the degree of $P$ may be much larger than $q$ (or $nq$, or even
$q^n$). To prove that
it is identically zero we need to use the fact that $P$ vanishes
{\em with high multiplicity} at every point in $\F_q^n$, and this
requires some multiplicity-enhanced version of the standard Schwartz-Zippel
lemma. We prove such a strengthening,
showing that the expected multiplicity of zeroes of a degree $d$
polynomial (even when $d \gg q$) at a random point in $\F_q^n$
is at most $d/q$ (see Lemma~\ref{lem:schwartz}).
Using this lemma, we are able to derive much better benefits
from the ``polynomial method''. Indeed we feel that this allows
us to fully utilize the power of the polynomial ring $\F_q[\vec X]$
and are not limited by the power of the function space
mapping $\F_q^n$ to $\F_q$.

Putting these ingredients together, the analysis of the Kakeya
sets follows easily. The analysis of the mergers follows
a similar path and may be viewed as a ``statistical'' extension
of the Kakeya set analysis to ``curve'' based sets, i.e., here
we consider sets $S$ that have the property that for a noticeable
fraction points $\vec x \in \F_q^n$ there exists a low-degree
curve passing through $\vec x$ that has a noticeable fraction
of its points in $S$. We prove such sets must also be large and
this leads to the analysis of the Dvir-Wigderson merger.

\paragraph{Organization of this paper.}
In Section~\ref{sec:prelims} we define the notion of the multiplicity
of the roots of a polynomial, using the notion of the
Hasse derivative. We present some basic facts about multiplicities
and Hasse derivatives, and also present the multiplicity based version
of the Schwartz-Zippel lemma.
In Section~\ref{sec:kakeya} we present our lower bounds for Kakeya sets.
In Section~\ref{sec:kakeyacurves} we extend this analysis for ``curves'' and
for ``statistical'' versions of the Kakeya property. This leads
to our analysis of the Dvir-Wigderson merger in Section~\ref{sec:merger}.
We then show how to use our mergers to construct the novel extractors
in Section~\ref{sec:extractor}.
Finally, in Section~\ref{asec:rs}, we include the algebraic proof of the
list-size bounds for the list-decoding of Reed-Solomon codes.

\paragraph{Version history.}
This version of the paper adds a new section (Section~\ref{sec:extractor})
constructing extractors based on the mergers given in the previous
version of this paper (dated 15 January 2009).

\section{Preliminaries}
\label{sec:prelims}

In this section we formally define the notion of
``mutliplicity of zeroes'' along with the companion
notion of the ``Hasse derivative''. We also describe
basic properties of these notions, concluding with
the proof of the ``multiplicity-enhanced version'' of
the Schwartz-Zippel lemma.

\subsection{Basic definitions}

We start with some notation.
We use $[n]$ to denote the set $\{1,\ldots,n\}$.
For a vector $\veci = \langle i_1,\ldots,i_n\rangle$
of non-negative integers, its {\em weight}, denoted
$\wt(\veci)$, equals $\sum_{j=1}^n i_j$.

Let $\mathbb F$ be any field, and $\F_q$ denote the
finite field of $q$ elements.
For $\vecx = \langle X_1,\ldots,X_n \rangle$, let $\F[\vecx]$
be the ring of polynomials in $X_1,\ldots,X_n$ with coefficients
in $\F$.
For a polynomial $P(\vec X)$, we let $H_{P}(\vec X)$ denote the homogeneous part of
$P(\vec X)$ of highest total degree.

For a vector of non-negative integers $\veci = \langle
i_1,\ldots,i_n \rangle$, let $\vecx^\veci$ denote the monomial
$\prod_{j= 1}^n X_{j}^{i_j} \in \F[\vecx]$. Note that the (total) degree of
this monomial equals $\wt(\veci)$.
For $n$-tuples of non-negative integers $\vec i$ and $\vec j$, we use
the notation
$${\vec i \choose \vec j} =
\prod_{k = 1}^n {i_k \choose j_k}.$$
Note that the coefficient of $\vec Z^{\vec i} \vec W^{\vec r - \vec i}$ in the expansion of $(\vec Z + \vec W)^{\vec r}$ equals ${ \vec r \choose \vec i }$.

\begin{definition}[(Hasse) Derivative]
For $P(\vec X) \in \F[\vec X]$ and
non-negative vector $\veci$, the $\veci$th {\em (Hasse) derivative} of
$P$, denoted $P^{(\vec i)}(\vec X)$,
is the coefficient of $\mathbf Z^{\veci}$ in the
polynomial $\tilde{P}(\vec X, \vec Z) \eqdef P(\vec X + \vec Z) \in \F[\vecx, \vec Z]$.
\end{definition}

Thus,
\begin{equation}
\label{eq:deridentity}
P(\vec X + \vec Z) = \sum_{\vec i} P^{(\veci)}(\vec X) \vec Z^{\vec i}.
\end{equation}

We are now ready to define the notion of the (zero-)multiplicity
of a polynomial at any given point.

\begin{definition}[Multiplicity]
For $P(\vec X) \in \F[\vec X]$ and $\vec a \in \F^n$,
the {\em multiplicity} of $P$ at $\veca \in \F^n$,
denoted $\mult(P,\veca)$, is the largest integer $M$ such that
for every
non-negative vector $\veci$ with $\wt(\veci) < M$, we have
$P^{(\veci)}(\veca) = 0$
(if $M$ may be taken
arbitrarily large, we set $\mult(P, \veca) = \infty$).
\end{definition}

Note that $\mult(P,\veca) \geq 0$ for every $\vec a$.
Also, $P(\vec a)=0$ if and only if $\mult(P,\veca) \geq 1$.

The above notations and definitions also extend naturally
to a tuple $P(\vec X) = \langle P_1(\vec X),\ldots,P_m(\vec X)
\rangle$ of polynomials with $P^{(\vec i)} \in \F[\vec X]^m$
denoting the vector
$\langle (P_1)^{(\vec i)},\ldots,(P_m)^{(\vec i)}\rangle$.
In particular, we define $\mult(P,\veca) = \min_{j \in [m]}
\{ \mult(P_j,\veca) \}$.

The definition of multiplicity above is similar to the
standard (analytic) definition of multiplicity with the difference
that the standard partial derivative has been replaced by the
Hasse derivative. The Hasse derivative is also a reasonably
well-studied
quantity (see, for example,~\cite[pages 144-155]{hirschfield})
and seems to have first appeared
in the CS literature
(without being explicitly referred to by this name)
in the work of Guruswami and Sudan~\cite{GuSu}.
It typically behaves like the
standard derivative, but with some key differences that make
it more useful/informative over finite fields. For completeness we
review basic properties of the Hasse derivative and multiplicity
in the following subsections.

\subsection{Properties of Hasse Derivatives}

The following proposition lists basic properties of the
Hasse derivatives. Parts (1)-(3) below are the same as
for the analytic derivative, while Part (4) is not!
Part (4) considers the derivatives of the derivatives
of a polynomial and shows a different relationship than
is standard for the analytic derivative. However crucial
for our purposes is that it shows that the $\vec j$th derivative
of the $\veci$th derivative is zero if (though not necessarily
only if)
the $(\veci + \vec j)$-th derivative is zero.

\begin{proposition}[Basic Properties of Derivatives]
\label{prop:prop}
Let $P(\vec X), Q(\vec X) \in \F[\vec X]^m$ and let $\vec i$, $\vec j$ be
vectors of nonnegative integers. Then:
\begin{enumerate}
\item $P^{(\veci)}(\vecx) + Q^{(\veci)}(\vecx) = (P+Q)^{(\veci)}(\vecx)$.
\item If $P$ is homogeneous of degree $d$, then $P^{(i)}$ is homogeneous of degree $d - \wt(i)$.
\item $(H_P)^{(\veci)}(\vecx) = H_{P^{(\veci)}}(\vecx)$
\item $\left(P^{(\veci)}\right)^{(\vec j)}(\vecx) = {\veci + \vec j \choose \veci}P^{(\mathbf{i+j})}(\vecx)$.
\end{enumerate}
\end{proposition}
\begin{proof}

Items 1 and 2 are easy to check, and item 3 follows immediately
from them. For item 4, we expand $P(\vec X + \vec Z + \vec W)$ in
two ways. First expand
\begin{eqnarray*}
P(\vec X + (\vec Z + \vec W)) &=& \sum_{\vec k} P^{(\vec k)}(\vec
X) (\vec Z + \vec W)^{\vec k} \\
&=& \sum_{\vec k} \sum_{\vec i + \vec j = \vec k} P^{(\vec
k)}(\vec X) { \vec k \choose \vec i} \vec
Z^{\vec j} \vec W^{\vec i}\\
&=& \sum_{\vec i, \vec j} P^{(\vec i + \vec j)}(\vec X) {\vec i +
\vec j \choose \vec i} \vec Z^{\vec j} \vec W^{\vec i} .
\end{eqnarray*}
On the other hand, we may write
$$P((\vec X + \vec Z) + \vec W) = \sum_{\vec i} P^{(\vec i)}(\vec X + \vec Z) \vec W^{\vec i} = \sum_{\vec i} \sum_{\vec j} \left( P^{(\vec i)}\right)^{(\vec j)}(\vec X) \vec Z^{\vec j} \vec W^{\vec i}.$$
Comparing coefficients of $\vec Z^{\vec j} \vec W^{\vec i}$ on both sides, we get the result.
\end{proof}


\subsection{Properties of Multiplicities}

We now translate some of the properties of the Hasse
derivative into properties of the multiplicities.

\begin{lemma}[Basic Properties of multiplicities]
\label{lem:dermult}
If $P(\vec X) \in \F[\vec X]$ and $\vec a \in \F^n$ are such that $\mult(P, \vec a) = m$, then $\mult(P^{(\vec i)}, \vec a) \geq m - \wt(\vec i)$.
\end{lemma}
\begin{proof}
By assumption, for any $\vec k$ with $\wt(\vec k) < m$, we have $P^{(\vec k)}(\vec a) = 0$.
Now take any $\vec j$ such that $\wt(\vec j) < m - \wt(\vec i)$.
By item 3 of Proposition~\ref{prop:prop}, $(P^{(\vec i)})^{(\vec j)}(\vec a) = {\vec i + \vec j \choose \vec i} P^{(\vec i + \vec j)}(\vec a)$. Since $\wt(\vec i + \vec j) = \wt(\vec i) + \wt(\vec j) < m$, we deduce that $(P^{(\vec i)})^{(\vec j)}(\vec a) = 0$.
Thus $\mult(P^{(\vec i)}, \vec a) \geq m - \wt(\vec i)$.
\end{proof}

We now discuss the behavior of multiplicities under composition of polynomial tuples.
Let $\vec X = (X_1, \ldots, X_n)$ and $\vec Y = (Y_1, \ldots, Y_\ell)$ be formal variables.
Let $P(\vec X) = (P_1(\vec X), \ldots, P_m(\vec X)) \in \F[\vec X]^m$ and $Q(\vec Y) = (Q_1(Y) ,\ldots, Q_n(Y)) \in \F[\vec Y]^n$. We define the
composition polynomial $P \circ Q(\vec Y) \in \F[\vec Y]^m$ to be the polynomial
$P(Q_1(\vec Y), \ldots, Q_n(\vec Y))$. In this situation we have the following proposition.

\begin{proposition}
\label{prop:multmult}
Let $P(\vec X), Q(\vec Y)$ be as above. Then for any $\vec a \in \F^{\ell}$,
$$ \mult(P \circ Q, \vec a) \geq \mult(P, Q(\vec a)) \cdot \mult(Q - Q(\vec a), \vec a).$$
In particular, since $\mult(Q - Q(\vec a), \vec a) \geq 1$, we have $\mult(P \circ Q, \vec a) \geq \mult(P, Q(\vec a))$.
\end{proposition}
\begin{proof}
Let $m_1 = \mult(P, Q(\vec a))$ and $m_2 = \mult(Q - Q(\vec a), \vec a)$.
Clearly $m_2 > 0$. If $m_1 = 0$ the result is obvious. Now assume $m_1 > 0$ (so that $P(Q(\vec a)) = 0$).
\begin{align*}
P(Q( \vec a + \vec Z)) &= P\left( Q(\vec a) + \sum_{\veci \neq 0} Q^{(\veci)}(\vec a)\vec Z^{\veci}\right)\\
&= P\left( Q(\vec a) + \sum_{\wt(\veci) \geq m_2} Q^{(\veci)}(\vec a)\vec Z^{\veci}\right) &\mbox{since $\mult(Q-Q(\vec a), \vec a) = m_2 > 0$}\\
&= P\left( Q(\vec a) + h(\vec Z)\right) &\mbox{where $h(\vec Z) = \sum_{\wt(\veci) \geq m_2} Q^{(\veci)}(\vec a) \vec Z^{\veci}$ }\\
&= P(Q(\vec a)) + \sum_{\vec j \neq 0} P^{(\vec j)}(Q(\vec a)) h(\vec Z)^{\vec j}\\
&= \sum_{\wt(\vec j) \geq m_1} P^{(\vec j)}(Q(\vec a)) h(\vec Z)^{\vec j}&\mbox{since $\mult(P, Q(\vec a)) = m_1 > 0$}
\end{align*}
Thus, since each monomial $\vec Z^{\veci}$ appearing in $h$ has $\wt(\veci) \geq m_2$, and each occurrence of
$h(\vec Z)$ in $P(Q(\veca + \vec Z))$ is raised to the power $\vec j$, with $\wt(\vec j) \geq m_1$, we conclude
that $P(Q(\veca + \vec Z))$ is of the form
$\sum_{\wt(\vec k) \geq m_1 \cdot m_2 } c_{\vec k} \vec Z^{\vec k}$.
This shows that $(P \circ Q)^{(\vec k)}(\vec a) = 0$ for each $\vec k$ with $\wt(\vec k) < m_1 \cdot m_2$, and
the result follows.
\end{proof}

\begin{corollary}
\label{corr:compose}
Let $P(\vec X) \in \F[\vec X]$ where $\vec X = (X_1, \ldots, X_n)$.
Let $\vec a, \vec b \in \F^n$.
Let $P_{\vec a, \vec b}(T)$ be the polynomial $P(\vec a + T \cdot \vec b) \in \F[T]$.
Then for any $t \in \F$, $$\mult(P_{\vec a, \vec b}, t) \geq \mult(P, \vec a + t \cdot \vec b).$$
\end{corollary}
\begin{proof}
Let $Q(T) = \vec a + T \vec b \in \F[T]^n$.
Applying the previous proposition to $P(\vec X)$ and $Q(T)$, we get the desired claim.
\end{proof}

\subsection{Strengthening of the Schwartz-Zippel Lemma}

We are now ready to state and prove the strengthening of the
Schwartz-Zippel lemma. In the standard form this lemma
states that the probability that $P(\vec a) = 0$ when
$\veca$ is drawn uniformly at random from $S^n$ is at most
$d/|S|$, where $P$ is a non-zero degree $d$ polynomial and
$S \subseteq \F$ is a finite set.
Using $\min\{1,\mult(P,\vec a)\}$ as the indicator variable
that is $1$ if $P(\vec a) = 0$, this lemma can be restated
as saying $\sum_{\veca\in S^n} \min\{1,\mult(P,\veca)\}
\leq d \cdot |S|^{n-1}$.
Our version below strengthens this lemma by replacing
$\min\{1,\mult(P,\veca)\}$ with $\mult(P,\veca)$ in this inequality.

\begin{lemma}
\label{lem:schwartz}
Let $P \in \F[\vecx]$ be a nonzero polynomial of total degree at most $d$.
Then for any finite $S \subseteq \F$,
$$\sum_{\veca \in S^n} \mult(P,\veca) \leq d\cdot |S|^{n-1}.$$
\end{lemma}
\begin{proof}
We prove it by induction on $n$.

For the base case when $n=1$, we first show that if $\mult(P,a) = m$
then $(X-a)^m$ divides $P(X)$. To see this, note that by definition
of multiplicity, we have that $P(a+Z) = \sum_{i} P^{(i)}(a) Z^i$
and $P^{(i)}(a) = 0$ for all $i < m$. We conclude that $Z^m$
divides $P(a+Z)$, and thus $(X-a)^m$ divides $P(X)$. It follows
that $\sum_{a \in S} \mult(P,a)$ is at most the degree of $P$.

Now suppose $n > 1$. Let $$P(X_1, \ldots, X_n) = \sum_{j = 0}^t
P_j(X_1, \ldots, X_{n-1}) X_n^j,$$ where $0 \leq t \leq d$,
$P_t(X_1, \ldots, X_{n-1}) \neq 0$ and $\deg(P_j) \leq d-j$.

For any $a_1, \ldots, a_{n-1} \in S$, let $m_{a_1, \ldots, a_{n-1}} = \mult(P_t, (a_1, \ldots, a_{n-1}))$. We will show that
\begin{equation}
\label{eq:schwartz}
\sum_{a_n \in S} \mult(P, (a_1, \ldots, a_n)) \leq m_{a_1, \ldots, a_{n-1}} \cdot |S| + t.
\end{equation}
Given this, we may then bound
$$ \sum_{a_1, \ldots, a_n \in S} \mult(P, (a_1, \ldots, a_n)) \leq \sum_{a_1, \ldots, a_{n-1} \in S} m_{a_1, \ldots, a_{n-1}}\cdot |S| + |S|^{n-1} \cdot t.$$
By the induction hypothesis applied to $P_t$, we know that
$$\sum_{a_1, \ldots, a_{n-1} \in S} m_{a_1, \ldots, a_{n-1}} \leq \deg(P_t) \cdot |S|^{n-2} \leq (d-t)\cdot
|S|^{n-2}.$$
This implies the result.

We now prove Equation~\eqref{eq:schwartz}. Fix $a_1, \ldots, a_{n-1} \in S$ and let $\vec i = (i_1, \ldots, i_{n-1})$ be such that $\wt(\vec i) = m_{a_1, \ldots, a_{n-1}}$ and $P_t^{(\vec i)}(X_1, \ldots, X_{n-1}) \neq 0$. Letting $(\vec i, 0)$ denote the vector $(i_1, \ldots, i_{n-1}, 0)$, we note that
$$P^{(\vec i, 0)}(X_1, \ldots, X_n) = \sum_{j = 0}^t P_j^{(\vec i)}(X_1, \ldots, X_{n-1}) X_n^j,$$
and hence $P^{(\vec i, 0)}$ is a nonzero polynomial.

Now by Lemma~\ref{lem:dermult} and Corollary~\ref{corr:compose}, we know that
\begin{align*}
\mult(P(X_1, \ldots, X_n), (a_1, \ldots, a_n)) &\leq \wt(\vec i, 0) + \mult( P^{(\vec i, 0)}(X_1, \ldots,X_n) , (a_1, \ldots, a_n)) \\
&\leq m_{a_1, \ldots, a_{n-1}} + \mult(P^{(\vec i, 0)}(a_1, \ldots, a_{n-1}, X_n), a_n) .
\end{align*}
Summing this up over all $a_n \in S$, and applying the $n=1$ case of this lemma to the nonzero univariate degree-$t$ polynomial $P^{(\vec i, 0)}(a_1, \ldots, a_{n-1}, X_n)$, we get Equation~\eqref{eq:schwartz}. This completes the proof of the lemma.
\end{proof}

The following corollary simply states the above lemma in
contrapositive form, with $S = \F_q$.

\begin{corollary}
\label{prop:zero}
Let $P \in \F_q[\vecx]$ be a polynomial of total degree at most $d$.
If $\sum_{\veca \in \F_q^n} \mult(P,\veca) > d\cdot q^{n-1}$,
then $P(\vecx) = 0$.
\end{corollary}

\section{A lower bound on the size of Kakeya sets}
\label{sec:kakeya}

We now give a lower bound on the size of Kakeya sets in $\F_q^n$.
We implement the plan described in Section~\ref{sec:intro}.
Specifically, in Proposition~\ref{prop:interpolate} we show that
we can find a somewhat low degree non-zero polynomial that vanishes
with
high multiplicity on any given Kakeya set, where the degree of
the polynomial grows with the size of the set.
Next, in Claim~\ref{clm:spread} we show that the homogenous part
of this polynomial vanishes with fairly high multiplicity
everywhere in $\F_q^n$. Using the strengthened Schwartz-Zippel
lemma, we conclude that the homogenous polynomial is identically
zero if the Kakeya set is too small, leading to the desired contradiction.
The resulting lower bound (slightly stronger than
Theorem~\ref{thm:kakeya-crude})
is given in Theorem~\ref{thm:kakeya}.

\begin{proposition}
\label{prop:interpolate}
Given a set $K \subseteq \F^n$ and non-negative integers
$m,d$ such that
$${m+n -1 \choose n} \cdot |K| < {d + n \choose n},$$ there
exists a non-zero polynomial $P = P_{m,K} \in \F[\vec X]$ of
total degree at most $d$ such that $\mult(P,\vec a) \geq m$
for every $\vec a \in K$.
\end{proposition}

\begin{proof}
The number of possible monomials in $P$ is ${d + n \choose n}$. Hence there
are ${d + n \choose n}$ degrees of freedom in the choice for the coefficients
for these monomials. For a given point $\vec a$, the condition that $\mult(P,\vec a) \geq m$
imposes ${m + n-1 \choose n}$ homogeneous linear constraints on the coefficients of $P$. Since
the total number of (homogeneous) linear constraints is ${m+n -1 \choose n} \cdot |K|$, which is
strictly less than the number of unknowns, there is a nontrivial solution.

\end{proof}

\begin{theorem}
\label{thm:kakeya}
If $K \subseteq \F_q^n$ is a Kakeya set, then
$|K| \geq \big(\frac{q}{2-1/q} \big)^n$.
\end{theorem}

\begin{proof}
Let $\ell$ be a large multiple of $q$ and let $$m = 2\ell -
\ell/q$$
 $$d = \ell q - 1.$$ These three parameters ($\ell,m$ and $d$)
 will be used as follows: $d$ will be the bound on the degree of
 a polynomial $P$ which vanishes on $K$, $m$ will be the
 multiplicity of the zeros of $P$ on $K$ and $\ell$
 will be the multiplicity of the zeros of the homogenous part of
 $P$ which we will deduce by restricting $P$ to lines passing
 through $K$.

Note that by the choices above we have $d < \ell q$ and $(m-\ell)
q > d - \ell$. We prove below later that
$$|K| \geq \frac{{d + n \choose n}}{{m+n-1 \choose n}} \geq
\alpha^n$$ where $\alpha \to \frac {q}{2-1/q}$ as $\ell \to
\infty$.

Assume for contradiction that $|K| < \frac{{d + n  \choose n}}{{m+n -1 \choose n}}$.
Then, by Proposition~\ref{prop:interpolate}
there exists a non-zero polynomial $P(\vecx) \in \F[\vecx]$
of total degree exactly $d^*$, where $d^* \leq d$, such that $\mult(P,\vec x) \geq m$ for
every $\vec x \in K$. Note that $d^* \geq \ell$ since
$d^* \geq m$ (since $P$ is nonzero and vanishes to multiplicity $\geq m$ at some point),
and $m \geq \ell$ by choice of $m$.
Let $H_P(\vecx)$ be the homogeneous part of $P(\vecx)$ of degree $d^*$. Note that
$H_P(\vecx)$ is nonzero.
The following claim shows that $H_P$ vanishes to multiplicity $\ell$ at each point of $\F_q^n$.



\begin{claim}
\label{clm:spread}
For each $\vecb \in \F_q^n$.
$$\mult(H_P, \vecb ) \geq \ell.$$
\end{claim}

\begin{proof}
Fix $\vec i$ with $\wt(\vec i) = w \leq \ell - 1$. Let $Q(\vec X) = P^{(\vec i)}(\vecx)$.
Let $d'$ be the degree of the polynomial $Q(\vecx)$, and note that $d' \leq d^* - w$.

Let $\veca = \veca(\vecb)$ be such that $\{\veca + t \vec b | t \in \F_q\} \subset K$.
Then for all $t \in \F_q$, by Lemma~\ref{lem:dermult}, $\mult(Q, \vec a + t\vec b) \geq m- w$.
Since $w \leq \ell -1$ and $(m- \ell) \cdot q > d^*- \ell $, we get that
$(m- w) \cdot q > d^*- w$.

Let $Q_{\vec a, \vec b}(T)$ be the polynomial $Q(\vec a + T\vec b)
\in \F_q[T]$. Then $Q_{\vec a, \vec b}(T)$ is a univariate
polynomial of degree at most $d'$, and by
Corollary~\ref{corr:compose}, it vanishes at each point of $\F_q$
with multiplicity $m- w$. Since $$(m- w) \cdot q > d^*- w \geq
\deg(Q_{\vec a, \vec b}(T)),$$ we conclude that $Q_{\vec a, \vec
b}(T) = 0$. Hence the coefficient of $T^{d'}$ in $Q_{\vec a, \vec
b}(T)$ is $0$. Let $H_Q$ be the homogenous component of $Q$ of
highest degree. Observe that the coefficient of $T^{d'}$ in
$Q_{\vec a, \vec b}(T)$ is $H_Q(\vec b)$. Hence $H_Q(\vec b) = 0$.


However $H_Q(\vec X) = (H_{P})^{(\vec i)}(\vec X)$ (by item 2 of Proposition~\ref{prop:prop}). Hence $(H_P)^{(\vec i)}(\vec b) = 0$. Since this is true for all $\vec i$ of weight at most
$ \ell - 1$, we conclude that $\mult(H_P,\vec b) \geq \ell$.
\end{proof}

Applying Corollary~\ref{prop:zero}, and noting that
$\ell q^n > d^* q^{n-1}$, we conclude that $H_{P}(\vec X) = 0$.
This contradicts the fact that $P(\vec X)$ is a nonzero polynomial.

Hence, $$|K| \geq \frac{{d + n \choose n}}{{m+n -1 \choose n}}$$

Now, by our choice of $d$ and $m$,
$$\frac{{d + n \choose n}}{{m+n -1 \choose n}}
= \frac{{\ell q - 1 + n \choose n}}{{2\ell - \ell/q +n -1 \choose n}}
= \frac{\prod_{i = 1}^n (\ell q - 1 + i)}{\prod_{i = 1}^n \left(2\ell - \ell/q -1 + i\right)}$$
Since this is true for all $\ell$ such that $\ell$ is a multiple of $q$, we get that
$$|K| \geq \lim_{\ell \to \infty}\prod_{i=1}^n \left(\frac {q - 1/l + i/l}{2 - 1/q - 1/l + i/l}\right) = \left(\frac{q}{2 - 1/q}\right)^n$$

\end{proof}

\section{Statistical Kakeya for curves}
\label{sec:kakeyacurves}

Next we extend the results of the previous section to a
form conducive to analyze the mergers of Dvir and
Wigderson~\cite{DW08}.
The extension changes two aspects of the consideration
in Kakeya sets, that we refer to as ``statistical'' and
``curves''. We describe these terms below.

In the setting of Kakeya sets we were given
a set $K$ such that for {\em every} direction, there was
a line in that direction such that {\em every} point on
the line was contained in $K$.
In the {\em statistical} setting we replace both occurrences
of the ``every'' quantifier with a weaker ``for many''
 quantifier. So we consider sets that satisfy the
condition that for many directions, there exists a line in that
direction intersecting $K$ in many points.

A second change we make is that we now consider curves
of higher degree and not just lines.
We also do not consider curves in various {\em directions}, but
rather curves passing through a given set of special points.
We start with formalizing the terms ``curves'', ``degree'' and
``passing through a given point''.

A {\em curve of degree $\Lambda$ in $\F_q^n$} is a tuple of polynomials $C(X) = (C_1(X), \ldots, C_n(X)) \in \F_q[X]^n$
such that $\max_{i \in [n]} \deg(C_i(X)) = \Lambda$. A curve $C$ naturally defines a map from $\F_q$ to $\F_q^n$.
For $\vec x \in \F_q^n$, we say that a curve $C$ {\em passes through $\vec x$} if there is a $t \in \F_q$ such that
$C(t) = \vec x$.

We now state and prove our statistical version of the Kakeya
theorem for curves.

\begin{theorem}[Statistical Kakeya for curves]
\label{lem:mergermain} Let $\lambda > 0,  \eta > 0$. Let $\Lambda > 0$
be an integer such that $\eta q > \Lambda$. Let $S \subseteq \F_q^n$ be
such that $|S| = \lambda q^n$. Let $K \subseteq \F_q^n$ be such
that for each $\vec x \in S$, there exists a curve $C_{\vec x}$ of
degree at most $\Lambda$ that passes through $\vec x$, and intersects
$K$ in at least $\eta q$ points. Then, $$|K| \geq
\left(\frac{\lambda q}{\Lambda \left(\frac{\lambda q - 1}{\eta q}\right)
+ 1}\right)^n.$$ In particular, if $\lambda \geq \eta $ we get
that $|K| \geq \left(\frac{\eta q}{\Lambda+1}\right)^n$.
\end{theorem}


Observe that when $\lambda = \eta = 1$, and $\Lambda = 1$, we get the same bound as that for Kakeya sets as obtained
in Theorem~\ref{thm:kakeya}.

\begin{proof}
Let $\ell$ be a large integer and let $$d = \lambda \ell q -1$$
$$m = \Lambda \frac{\lambda \ell q -1 -(\ell-1)}{\eta q} + \ell.$$ By
our choice of $m$ and $d$, we have $\eta q(m-(\ell-1)) > \Lambda(d -
(\ell-1))$. Since $\eta q > \Lambda$, we have that for all $w$ such that
$0 \leq w \leq \ell-1$, $\eta q(m-w) > \Lambda(d - w)$. Just as in the
proof of Theorem~\ref{thm:kakeya}, we will prove that $$|K| \geq
\frac{{d + n \choose n}}{{m+n-1 \choose n}} \geq \alpha^n$$ where
$\alpha \to \frac{\lambda q}{\Lambda \left(\frac{\lambda q - 1}{\eta
q}\right) + 1}$ as $\ell \to \infty$.

If possible, let $|K| < \frac{{d + n \choose n}}{{m+n-1 \choose n}}$.
As before, by Proposition~\ref{prop:interpolate}
there exists a non-zero polynomial $P(\vecx) \in \F_q[\vecx]$
of total degree $d^*$, where $d^* \leq d$, such that $\mult(P,\veca) \geq m$ for
every $\veca \in K$. We will deduce that in fact $P$ must vanish on all points in $S$ with multiplicity
$\ell$. We will then get the desired contradiction from Corollary~\ref{prop:zero}.

\begin{claim}
For each $\vec x_0 \in S$,
$$\mult(P, \vec x_0) \geq \ell.$$
\end{claim}
\begin{proof}
Fix any $\vec i$ with $\wt(\vec i) = w \leq  \ell -1$. Let $Q(\vec X) = P^{(\vec i)}(\vecx)$.
Note that $Q(\vecx)$ is a polynomial of degree at most $d^* - w$.
By Lemma~\ref{lem:dermult}, for all points $\vec a \in K$, $\mult(Q, \vec a) \geq m-w$.

Let $C_{\vec x_0}$ be the curve of degree $\Lambda$ through $\vec x_0$, that
intersects $K$ in at least $\eta q$ points. Let $t_0 \in \F_q$ be such that $C_{\vec x_0}(t_0) = \vec x_0$.
Let $Q_{\vec x_0}(T)$ be the polynomial $Q \circ C_{\vec x_0}(T) \in \F_q[T]$. Then $Q_{\vec x_0}(T)$ is a univariate
polynomial of degree at most $\Lambda(d^* - w)$.
By Corollary~\ref{corr:compose}, for all points $t \in \F_q$ such that $C_{\vec x_0}(t) \in K$,  $Q_{\vec x_0}(T)$ vanishes at $t$ with multiplicity $m-w$. Since the number of such points $t$ is at least $\eta q$, we get that $Q_{\vec x_0}(T)$ has at least $\eta q(m-w)$ zeros (counted with multiplicity).
However, by our choice of parameters, we know that
$$\eta q(m-w) > \Lambda(d - w) \geq \Lambda(d^* -w) \geq \deg(Q_{\vec x_0}(T)).$$ Since the degree of  $Q_{\vec x_0}(T)$ is strictly less than the number of its
zeros, $Q_{\vec x_0}(T)$ must be identically zero. Thus we get $Q_{\vec x_0}(t_0) = Q(C_{\vec x_0}(t_0)) = Q(\vec x_0) = 0$
Hence $P^{(\vec i)}(\vec x_0) = 0$. Since this is true for all $\vec i$ with $\wt(\vec i) \leq \ell -1$,
we conclude that $\mult(P, \vec x_0) \geq \ell$.
\end{proof}

Thus $P$ vanishes at every point in $S$ with multiplicity $\ell$.
As $P(\vec X)$ is a non-zero polynomial, Corollary~\ref{prop:zero} implies that $\ell |S| \leq d^* q^{n-1}$. Hence
$\ell \lambda q^n \leq d q^{n-1}$, which contradicts the choice of $d$.

Thus $|K| \geq \frac{{d + n \choose n}}{{m+n-1 \choose n}}$.
By choice of $d$ and $m$, $$|K| \geq \frac{{\lambda \ell q -1 + n \choose n}}
{{\Lambda \frac{\lambda \ell q -1 -(\ell-1)}{\eta q} + \ell +n-1 \choose n}}.$$

Picking $\ell$ arbitrarily large, we conclude that $$|K| \geq \lim_{\ell \to \infty} \frac{{\lambda \ell q -1 + n \choose n}} {{\Lambda \frac{\lambda \ell q -1 -(\ell-1)}{\eta q} + \ell +n-1 \choose n}}
= \lim_{\ell \to \infty} \left(\frac{\ell \lambda q -1}{\ell \Lambda \left(\frac{\lambda q - 1}{\eta q}\right) + \ell}\right)^n = \left(\frac{\lambda q}{\Lambda \left(\frac{\lambda q - 1}{\eta q}\right) + 1}\right)^n.$$
\end{proof}

\section{Improved Mergers}
\label{sec:merger}

In this section we state and prove our main result on
randomness mergers.

\subsection{Definitions and Theorem Statement}

We start by recalling some basic quantities associated with
random variables.
The {\sf statistical distance} between two random variables $\X$ and
$\Y$ taking values from a finite domain $\Omega$ is defined as
\[\mathop{\max}_{S\subseteq \Omega} \left|\prob[\X \in S] - \prob[
\Y \in S]\right|
.\] We say that $\X$ is $\epsilon$-{\sf close} to $\Y$ if the
statistical distance between $\X$ and $\Y$ is at most $\epsilon$,
otherwise we say that $\X$ and $\Y$ are $\epsilon$-{\sf far}.  The
{\sf min-entropy} of a random variable $\X$ is defined as
\[  H_{\infty}(\X) \triangleq \min_{x \in \text{supp}(\X)}
\log_2\left(\frac{1}{\prob[\X=x]}\right). \]
We say that a random variable $\X$ is $\epsilon$-close
to having min-entropy $m$ if there exists a random
variable $\Y$ of min-entropy $m$ such that $\X$ is
$\epsilon$-close to $\Y$.

A ``merger'' of randomness takes a $\Lambda$-tuple of random variables
and ``merges'' their randomness to produce a high-entropy
random variable, provided the $\Lambda$-tuple is ``somewhere-random'' as defined below.

\begin{definition}[Somewhere-random source]
For integers $\Lambda$ and $N$
a {\em simple $(N,\Lambda)$-somewhere-random source} is a random variable
$\A = (\A_1, \ldots, \A_\Lambda)$ taking values in $S^\Lambda$,
where $S$ is some finite set of cardinality $2^N$,
such that for some $i_0 \in [\Lambda]$,
the distribution of $\A_{i_0}$ is uniform over
$S$. A {\em $(N, \Lambda)$-somewhere-random source} is a convex combination
of simple $(N, \Lambda)$-somewhere-random sources.
(When $N$ and $\Lambda$ are clear from context we refer to the source
as simply a ``somewhere-random source''.)
\end{definition}

We are now ready to define a merger.

\begin{definition}[Merger]
For positive integer $\Lambda$ and set $S$ of size $2^N$,
a function $f : S^\Lambda \times \{0,1\}^d \rightarrow S$
is called an {\em $(m,\epsilon)$-merger} (of $(N,\Lambda)$-somewhere-random sources),
if for every $(N,\Lambda)$ somewhere-random
source $\A = (\A_1, \ldots, \A_\Lambda)$ taking values in
$S^\Lambda$, and for $\B$ being uniformly distributed over $\{0,1\}^d$,
the distribution of $f((\A_1, \ldots, \A_\Lambda), \B)$
is $\epsilon$-close to having min-entropy $m$.
\end{definition}

A merger thus has five parameters associated with it: $N$, $\Lambda$,
$m$, $\epsilon$ and $d$. The general goal is to give explicit
constructions of mergers of $(N,\Lambda)$-somewhere-random sources for
every choice of $N$ and $\Lambda$, for as large an $m$ as possible, and
with $\epsilon$ and $d$ being as small as possible. Known mergers
attain $m = (1-\delta)\cdot N$ for arbitrarily small $\delta$ and
our goal will be to achieve $\delta = o(1)$ as a function of $N$,
while $\epsilon$ is an arbitrarily small positive real number.
Thus our main concern is the growth of $d$ as a function of $N$
and $\Lambda$. Prior to this work, the best known bounds required either
$d = \Omega(\log N + \log \Lambda)$ or $d = \Omega(\Lambda)$. We only require
$d = \Omega(\log \Lambda)$.

\begin{theorem}
\label{thm:themerger}
For every $\epsilon,\delta > 0$ and integers $N, \Lambda$,
there exists a $((1-\delta)\cdot N, \epsilon)$-merger
of $(N,\Lambda)$-somewhere-random sources, computable in
polynomial time, with seed length
$$d = \frac1\delta \cdot \log_2 \left( \frac{2\Lambda}{\epsilon} \right).$$
\end{theorem}

\subsection{The Curve Merger of \cite{DW08} and its analysis}

The merger that we consider is a very simple one proposed
by Dvir and Wigderson~\cite{DW08}, and we improve their
analysis using our extended method of multiplicities.
We note that they used the polynomial method in their
analysis; and the basic method of multiplicities doesn't seem
to improve their analysis.

The curve merger of \cite{DW08}, denoted $f_{\DW}$, is obtained as
follows. Let $q\geq \Lambda$ be a prime power, and let $n$ be any
integer. Let $\gamma_1, \ldots, \gamma_\Lambda \in \F_q$ be distinct,
and let $c_i(T) \in \F_q[T]$ be the unique degree $\Lambda-1$ polynomial
with $c_i(\gamma_i) = 1$ and for all $j \neq i$, $c_i(\gamma_j) =
0$. Then for any $\vec x = (\vec {x_1}, \ldots, \vec {x_\Lambda}) \in
(\F_q^n)^\Lambda$ and $u \in \F_q$, the curve merger $f_{\DW}$ maps
$(\F_q^n)^\Lambda \times \F_q$ to $\F_q^n$ as follows:
$$f_{\DW}((\vec {x_1}, \ldots, \vec {x_\Lambda}), u) = \sum_{i= 1}^{\Lambda} c_i(u) \vec{x_i}.$$
In other words,
$f_{\DW}((\vec {x_1}, \ldots, \vec {x_\Lambda}), u)$
picks the (canonical) curve passing through
$\vec {x_1},\ldots,\vec {x_\Lambda}$ and outputs the $u$th point on
the curve..

\begin{theorem}
\label{thm:clean} Let $q \geq \Lambda$ and $\A$ be somewhere-random
source taking values in $(\F_q^n)^\Lambda$. Let $\B$ be distributed
uniformly over $\F_q$, with $\A, \B$ independent. Let $\C = f_{\DW}(\A, \B)$. Then for $$q
\geq \left(\frac {2\Lambda}{\epsilon}\right)^{\frac{1}{\delta}},$$ $\C$
is $\epsilon$-close to having min-entropy $(1-\delta) \cdot n
\cdot \log_2 q$.
\end{theorem}

Theorem~\ref{thm:themerger} easily follows from the above.
We note that \cite{DW08} proved a similar theorem assuming
$q \geq \poly(n,\Lambda)$, forcing their seed length to grow logarithmically
with $n$ as well.

\begin{proofof}{Theorem~\ref{thm:themerger}}
Let $q = 2^d$, so that $q \geq \left(\frac {2\Lambda}{\epsilon}\right)^{\frac{1}{\delta}}$, and let $n = N/d$. Then we may identify identify $\F_q$ with $\{0,1\}^d$ and $\F_q^n$ with $\{0,1\}^N$.
Take $f$ to be the function $f_{\DW}$ given earlier. Clearly $f$ is computable in the claimed time.
Theorem~\ref{thm:clean} shows that $f$ has the required merger property.
\end{proofof}

We now prove Theorem~\ref{thm:clean}.

\begin{proofof}{Theorem~\ref{thm:clean}}
Without loss of generality, we may assume that $\A$ is a simple somewhere-random source.
Let $m = \left( 1- \delta \right)\cdot n\cdot \log_2 q$.
We wish to show that $f_{\DW}(\A, \vec \B)$ is $\epsilon$-close to having min-entropy $m$.

Suppose not. Then there is a set $K \subseteq \F_q^{n}$ with $|K| \leq 2^m = q^{(1-\delta)\cdot n} \leq \left(\frac{\epsilon q}{2\Lambda}\right)^{n}$ such that
$$\Pr_{\A, \B} [f(\A, \B) \in K] \geq \epsilon.$$

Suppose $\A_{i_0}$ is uniformly distributed over $\F_q^n$. Let
$\A_{-i_0}$ denote the random variable $$(\A_1, \ldots,
\A_{i_0-1}, \A_{i_0+1}, \ldots, \A_\Lambda).$$ By an averaging
argument, with probability at least $\lambda = \epsilon/2$ over
the choice of $\A_{i_0}$, we have
$$\Pr_{\A_{-i_0}, \B} [f(\A, \B) \in K] \geq \eta,$$
where $\eta = \epsilon/2$.
Since $\A_{i_0}$ is uniformly distributed over $\F_q^n$, we conclude that there is a set $S$ of
cardinality at least $\lambda q^n$ such that for any
$\vec x \in S$,
$$\Pr_{\A, \B } [f( \A, \B) \in K \mid \A_{i_0} = \vec x] \geq \eta.$$
Fixing the values of $\A_{-i_0}$, we conclude that for each $\vec x \in S$,
there is a $\vec y = \vec y(\vec x) = (\vec y_1, \ldots, \vec y_\Lambda)$ with $\vec y_{i_0} = x$ such that $\Pr_{\B}[f(\vec y, \B) \in K] \geq \eta$.
Define the degree $\Lambda-1$ curve $C_{\vec x}(T) = f(\vec y(\vec x), T) = \sum_{j=1}^\Lambda \vec y_j c_j(T)$.
Then $C_{\vec x}$ passes through $\vec x$, since $C_{\vec x}(\gamma_{i_0}) = \sum_{j=1}^\Lambda \vec y_j c_j(\gamma_{i_0}) = \vec y_{i_0} = \vec x$, and $\Pr_{\B \in \F_q}[C_{\vec x}(\B) \in K] \geq \eta$ by definition of $C_{\vec x}$.

Thus $S$ and $K$ satisfy the hypothesis of Theorem~\ref{lem:mergermain}. We now conclude that
$$|K| \geq \left(\frac{\lambda q}{(\Lambda-1) \left(\frac{\lambda q - 1}{\eta q}\right) + 1}\right)^n = \left(\frac{\epsilon q/2}{ \Lambda - (\Lambda-1)/\eta q}\right)^n > \left(\frac{\epsilon q}{2\Lambda} \right)^n.
$$ This is a contradiction, and the proof of the theorem is complete.
\end{proofof}

\paragraph{The Somewhere-High-Entropy case:}
It is possible to extend the merger analysis given above also to
the case of {\em somewhere-high-entropy} sources. In this scenario
the source is comprised of blocks, one of which has min entropy at
least $r$. One can then prove an analog of Theorem~\ref{thm:clean}
saying that the output of $f_{DW}$ will be close to having min
entropy $(1-\delta)\cdot r$ under essentially  the same conditions
on $q$. The proof is done by hashing the source using a random
linear function into a smaller dimensional space and then applying
Theorem~\ref{thm:clean} (in a black box manner). The reason why
this works is that the merger commutes with the linear map (for
details see \cite{DW08}).

\section{Extractors with sub-linear entropy loss}
\label{sec:extractor}

In this section we use our improved analysis of the Curve Merger
to show the existence of an explicit extractor
with logarithmic seed and sub linear entropy loss.

We will call a random variable $\X$ distributed over $\omm^n$ with
min-entropy $k$ an $(n,k)$-{\sf source}.

\begin{definition}[Extractor]
A function $E : \omm^n \times \omm^d \mapsto \omm^m$ is a
$(k,\eps)$-{\sf extractor} if for every $(n,k)$-source $\X$, the
distribution of $E(\X,\U_d)$ is $\eps$-close to uniform, where $\U_d$
is a random variable distributed uniformly over $\omm^d$, and $\X, \U_d$
are independent. An extractor
is called {\em explicit} if it can be computed in polynomial time.
\end{definition}

It is common to refer to the quantity $k-m$ in the above
definition as the {\em entropy loss} of the extractor. The next
theorem asserts the existence of an explicit extractor with
logarithmic seed and sub-linear entropy loss.

\begin{theorem}[Basic extractor with sub-linear entropy
loss]\label{thm-extractor} For every $c_1 \geq 1$, for all
positive integers $k <n$ with $k \geq \log^2(n)$, there exists an
explicit $(k,\eps)$-extractor $E : \omm^n \times \omm^d \mapsto
\omm^m$ with
\[ d = O(c_1 \cdot\log(n)), \]
\[ k-m = O\left( \frac{k\cdot \log\log(n)}{\log(n)} \right), \]
\[ \eps = O\left( \frac{1}{\log^{c_1}(n) } \right). \]
\end{theorem}

The extractor of this theorem is constructed by composing several
known explicit constructions of pseudorandom objects with the
merger of Theorem~\ref{thm:themerger}. In Section~\ref{sec-basic}
we describe the construction of our basic extractor. We then show,
in Section~\ref{sec-final} how to use the 'repeated extraction'
technique of Wigderson and Zuckerman \cite{WZ} to boost this
extractor and reduce the entropy loss to $k-m = O(k/\log^{c}n)$
for any constant $c$ (while keeping the seed logarithmic). The end
result is the following theorem:

\begin{theorem}[Final extractor with sub-linear entropy
loss]\label{thm-final-extractor} For every $c_1, c_2 \geq 1$, for
all positive integers $k < n$, there exists an explicit
$(k,\eps)$-extractor $E : \omm^n \times \omm^d \mapsto \omm^m$
with
\[ d = O(c_1c_2 \cdot \log(n)), \]
\[ k-m = O\left( \frac{k}{\log^{c_2}(n)} \right), \]
\[ \eps = O\left( \frac{1}{\log^{c_1}(n) } \right). \]
\end{theorem}

\subsection{Proof of Theorem~\ref{thm-extractor}}\label{sec-basic}

Note that we may equivalently view an extractor $E : \{0,1\}^n
\times \{0,1\}^d \rightarrow \{0,1\}^m$ as a randomized algorithm
$E : \{0, 1\}^n \rightarrow \{0,1\}^m$ which is allowed to use $d$
uniformly random bits. We will present the extractor $E$ as such
an algorithm which takes 5 major steps.

Before giving the formal proof we give a high level description of
our extractor. Our first step is to apply the lossless condenser
of \cite{GUV} to output a string of length $2k$ with min entropy
$k$ (thus reducing our problem to the case $k = \Omega(n)$). The
construction continues along the lines of \cite{DW08}. In the
second step, we partition our source (now of length $n' = 2k$)
into $\Lambda = \log(n)$ consecutive blocks $X_1,\ldots,X_\Lambda
\in \omm^{n'/\Lambda}$ of equal length. We then consider the
$\Lambda$ possible divisions of the source into a prefix of $j$
blocks and suffix of $\Lambda-j$ blocks for $j$ between $1$ and
$\Lambda$. By a result of Ta-Shma~\cite{TS}, after passing to a
convex combination, one of these divisions is a $(k',k_2)$ block
source with $k'$ being at least $k - O(k/\Lambda)$ and $k_2$ being
at least poly-logarithmic in $k$. In the third step we use a block
source extractor (from \cite{RSW00}) on each one of the possible
$\Lambda$ divisions (using the same seed for each division) to
obtain a somewhere random source with block length $k'$. The
fourth step is to merge this somewhere random source into a single
block of length $k'$ and entropy $k'\cdot (1-\delta)$ with
$\delta$ sub-constant. In view of our new merger parameters, and
the fact that $\Lambda$ (the number of blocks) is small enough, we
can get away with choosing $\delta = \log\log(n) / \log(n)$ and
keeping the seed logarithmic and the error poly-logarithmic. To
finish the construction (the fifth step) we need to extract almost
all the entropy from a source of length $k'$ and entropy $k' \cdot
(1 - \delta)$. This can be done (using known techniques) with
logarithmic seed and  an additional entropy loss of $O(\delta
\cdot k')$.

We now formally prove Theorem~\ref{thm-extractor}. We begin by
reducing to the case where $n = O(k)$ using the lossless
condensers of \cite{GUV}.
\begin{theorem}[Lossless condenser \cite{GUV}]\label{thm-lossless}
For all integers positive $k < n$ with $k = \omega(\log(n))$, there exists an explicit function
$C_{\GUV} : \omm^n \times \omm^{d'} \mapsto \omm^{n'}$ with $n' = 2k$, $d' =
O(\log(n))$, such that for every $(n,k)$-source $\X$,
$C(\X, \U_{d'})$ is $(1/n)$-close to an $(n', k)$-source,
where $\U_{d'}$ is distributed uniformly over $\omm^{d'}$, and
$\X, \U_{d'}$ are independent.
\end{theorem}

\fbox{{\bf Step 1:} Pick $\U_{d'}$ uniformly from $\omm^{d'}$. Compute $\X' = C_{\GUV}(\X, \U_{d'})$.}

By the above theorem, $\X'$ is $(1/n)$-close to an $(n',
k)$-source, where $n' = 2k$. Our next goal is to produce a {\em
somewhere-block source}. We now define these formally.

\begin{definition}[Block Source]
Let $\X = (\X_1,\X_2)$ be a random source over $\omm^{n_1} \times
\omm^{n_2}$. We say that $\X$ is a $(k_1,k_2)$-{\sf block source}
if $\X_1$ is an $(n_1,k_1)$-source and for each $x_1 \in
\omm^{n_1}$ the conditional random variable $\X_2 | \X_1 = x_1$ is
an $(n_2,k_2)$-source.
\end{definition}

\begin{definition}[Somewhere-block source]
Let $\X = (\X_1,\ldots,\X_\Lambda)$ be a random variable such that
each $\X_i$ is distributed over $\omm^{n_{i,1}}\times
\omm^{n_{i,2}}$. We say that $\X$ is a {\sf simple
$(k_1,k_2)$-somewhere-block source} if there exists $i \in
[\Lambda]$ such that $\X_i$ is a $(k_1,k_2)$-block source. We say
that $\X$ is a {\sf somewhere-$(k_1,k_2)$-block source} if $\X$ is
a convex combination of simple somewhere random sources.
\end{definition}






We now state a result of Ta-Shma~\cite{TS} which converts an
arbitrary source into a somewhere-block source. This is the first
step in the proof of Theorem 1 on Page 44 of \cite{TS} (Theorem 1
shows how convert any arbitrary source to a somewhere-block
source, and then does more by showing how one could extract from
such a source).

Let $\Lambda$ be an integer and assume for simplicity of notation
that $n'$ is divisible by $\Lambda$. Let
\[ \X' = (\X'_1,\ldots,\X'_\Lambda) \in \left( \omm^{n'/\Lambda} \right)^\Lambda \]
denote the partition of $\X'$ into $\Lambda$ blocks. For every $1
\leq j < \Lambda$ we denote
\[ \Y_j = (\X'_1,\ldots,\X'_j), \]
\[ \ZZ_j = (\X'_{j+1},\ldots,\X'_{\Lambda}), \]

Consider the function $B^{\Lambda}_{\TS} : \{0,1\}^{n'} \rightarrow
(\omm^{n'})^{\Lambda}$, where $$B^{\Lambda}_{\TS}(X') = ((\Y_1,
\ZZ_1), (\Y_2, \ZZ_2), \ldots, (\Y_{\Lambda }, \ZZ_{\Lambda })).$$
The next theorem shows that the source $\left( (\Y_j,\ZZ_j)
\right)_{j \in [\Lambda]}$ is close to a somewhere-block source.

\begin{theorem}[\cite{TS}]\label{thm-TS}
Let $\Lambda$ be an integer. Let $k = k_1 + k_2 + s$. Then the
function $B^{\Lambda}_{\TS} : \{0,1\}^{n'} \rightarrow
(\omm^{n'})^{\Lambda}$ is such that for any $(n', k)$-source $\X'$,
letting $\X'' = B^{\Lambda}_{\TS}(\X')$, we have that $\X''$ is
$O(n\cdot2^{-s})$-close to a somewhere-$(k_1 - O(n'/\Lambda),
k_2)$-block source.
\end{theorem}

\fbox{ {\bf Step 2:} Set $\Lambda = \log(n)$. Compute $X'' =
(\X''_1, \X''_2, \ldots, \X''_\Lambda) = B^{\Lambda}_{\TS}(\X')$. }

Plugging $k_2 = O(\log^4(n')) = O(\log^4(k))$, $s = O(\log n)$ and
$k_1 = k - k_2 -s$ in the above theorem, we conclude that $\X''$
is $n^{-\Omega(1)}$-close to a somewhere-$(k', k_2)$-block source,
where $$k' = k_1 - O(n'/\log(n)) = k - k_2 - s - O(k/\log(n)) = k
- O(k/ \log(n)),$$ where for the last inequality we use the fact
that $k > log^2(n)$ and so both $s$ and $k_2$ are bounded by
$O(k/\log(n))$.

We next use the block source extractor from~\cite{RSW00} to convert
the above somewhere-block source to a somewhere-random source.

\begin{theorem}[\cite{RSW00}]\label{thm-RSW}
Let $n' = n_1 + n_2$ and let $k',k_2$ be such that $k_2 >
\log^4(n_1)$. Then there exists an explicit function $E_{\RSW} :
\omm^{n_1} \times \omm^{n_2} \times \omm^{d''} \mapsto \omm^{m''}$
with $m'' = k'$, $d'' = O(\log(n'))$, such that for any $(k',
k_2)$-block source $\X$, $E_{\RSW}(\X, \U_{d''})$ is
$(n_1)^{-\Omega(1)}$-close to the uniform distribution over
$\omm^{m''}$, where $\U_{d''}$ is distributed uniformly over
$\omm^{d''}$, and $\X, \U_{d''}$ are independent.
\end{theorem}

Set $d'' = O( log(n'))$ as in Theorem~\ref{thm-RSW}.

\fbox{ {\bf Step 3:} Pick $\U_{d''}$ uniformly from $\omm^{d''}$.
For each $j \in [\Lambda]$, compute $\X'''_j = E_{\RSW}(\X''_j,
U_{d''})$. }

By the above theorem, $\X'''$ is ${n'}^{-\Omega(1)}$-close to a
somewhere-random source. We are now ready to use the merger $M$
from Theorem~\ref{thm:themerger}. We invoke that theorem with
entropy-loss $\delta = \log\log(n)/\log(n)$ and error $\epsilon =
\frac{1}{\log^{c_1}(n)}$, and hence $M$ has a seed length of
$$d''' = O(\frac{1}{\delta} \log \frac{\Lambda}{\epsilon}) = O(c_1
\log(n) ).$$

\fbox{
{\bf Step 4:} Pick $\U_{d'''}$ uniformly from $\omm^{d'''}$. Compute $\X'''' = M(\X''', \U_{d'''})$.
}

By Theorem~\ref{thm:themerger}, $\X''''$ is $O(\frac{1}{\log^{c_1}(n)})$-close to a
$(k', (1-\delta) k')$-source. Note that $\delta = o(1)$, and thus $\X''''$ has
nearly full entropy. We now apply an extractor for sources with extremely-high
entropy rate, given by the following lemma.
\begin{lemma}\label{lem-exthigh}
For any $k'$ and $\delta > 0$, there exists an explicit
$(k'(1-\delta), k'^{-\Omega(1)})$-extractor $E_{\HIGH} : \omm^{k'}
\times \omm^{d''''} \mapsto \omm^{(1-3\delta)k'}$ with $d'''' =
O(\log(k'))$.
\end{lemma}
The proof of this lemma follows easily from Theorem~\ref{thm-RSW}.
Roughly speaking, the input is partitioned into blocks of length
$k' - \delta k - \log^4 k'$ and $\delta k' + \log^4 k'$. It
follows that this partition is close to a $(k' (1-2\delta) -
\log^4 k', \log^4 k')$-block source. This block source is then
passed through the block-source extractor of
Theorem~\ref{thm-RSW}.

\fbox{ {\bf Step 5:} Pick $\U_{d''''}$ uniformly from
$\omm^{d''''}$. Compute $\X''''' = E_{\HIGH}(\X'''', \U_{d''''})$.
{\bf \it Output $\X'''''$}. }

This completes the description of the extractor $E$. It remains to
note that $d$, the total number of random bits used, is at most
$d' + d'' + d''' + d'''' = O(c_1 \log n)$, and that the output
$\X'''''$ is $O(\frac{1}{\log^{c_1} n})$-close to uniformly
distributed over $$\omm^{(1-3 \delta)k'} = \omm^{k - O(k \cdot
\frac{\log \log n}{\log n})}.$$ This completes the proof of
Theorem~\ref{thm-extractor}.

We summarize the transformations in the following table:
\begin{center}
\begin{tabular}{|c|c|c|c|}
\hline
Function & Seed length & Input-type & Output-type\\
\hline
$C_{\GUV}$ & $O(\log(n))$ & $(n,k)$-source & $(2k, k)$-source\\
$B^{\Lambda}_{\TS}$ &  0 & $(2k, k)$-source & somewhere-$(k', \log^4(k))$-block \\
$E_{\RSW}$ & $O(\log(k))$ & somewhere-$(k', \log^4(k))$-block & $(k', O(\log(n)))$-somewhere-random\\
$M$ & $O(\log (n))$ &  $(k', O(\log(n)))$-somewhere-random & $(k', k' - o(k))$-source\\
$E_{\HIGH}$ & $O(\log (k))$ & $(k', k'-o(k))$-source & $\U_{k'-o(k)}$\\
\hline
\end{tabular}
\end{center}


\subsection{Improving the output length by repeated extraction}\label{sec-final}
We now use some ideas from \cite{RSW00} and \cite{WZ} to extract
an even larger fraction of the min-entropy out of the source. This
will prove Theorem~\ref{thm-final-extractor}. We first prove a
variant of the theorem with a restriction on $k$. This restriction
will be later removed using known constructions of extractors for
low min-entropy.

\begin{theorem}[Explicit extractor with improved sub-linear entropy
loss]\label{thm-better-extractor} For every $c_1, c_2 \geq 1$, for all positive integers $k < n$ with $k = \log^{\omega(1)}(n)$,
there exists an explicit $(k,\eps)$-extractor $E : \omm^n \times \omm^d \mapsto
\omm^m$ with
\[ d = O(c_1c_2 \cdot \log(n)), \]
\[ k-m = O\left( \frac{k}{\log^{c_2}(n)} \right), \]
\[ \eps = O\left( \frac{1}{\log^{c_1}(n) } \right). \]
\end{theorem}

We first transform the extractor given in
Theorem~\ref{thm-extractor} into  a {\it strong} extractor
(defined below) via \cite[Theorem 8.2]{RSW00} (which gives a
generic way of getting a strong extractor from any extractor). We
then use a trick from \cite{WZ} that repeatedy uses the same
extractor with independent seeds to extract the `remaining
entropy' from the source, thus improving the entropy loss.

\begin{definition} A  $(k,\eps)$-{\sf extractor} $E : \omm^n \times \omm^d \mapsto \omm^m$ is strong if for every $(n,k)$-source $\X$, the distribution of $(E(\X, \U_d), \U_d)$ is $\eps$-close to the uniform distribution over $\omm^{m+d}$,
where $\U_d$ is distributed uniformly over $\omm^d$, and $\X, \U_d$ are independent.
\end{definition}

\begin{theorem}{\bf(\cite[Theorem 8.2]{RSW00})}\label{thm-rsw}
Any explicit $(k,\eps)$-{\sf extractor} $E : \omm^n \times \omm^d \mapsto \omm^m$ can be transformed into an explicit strong $(k,O(\sqrt{\eps}))$-{\sf extractor}  $E' : \omm^n \times \omm^{O(d)} \mapsto \omm^{m-d-2\log(1/\eps) -O(1)}$.
\end{theorem}

\begin{theorem}{\bf(\cite[Lemma 2.4]{WZ})}\label{thm-wz}
Let $E_1 : \omm^n \times \omm^{d_1} \mapsto \omm^{m_1}$ be an
explicit strong $(k,\eps_1)$-{\sf extractor}, and let  $E_2 :
\omm^n \times \omm^{d_2} \mapsto \omm^{m_2}$ be an explicit strong
$(k - (m_1 + r),\eps_2)$-{\sf extractor}. Then the function $$E_3
: \omm^n \times \left( \omm^{d_1} \times \omm^{d_2} \right)
\mapsto \omm^{m_1+m_2}$$ defined by
$$E_3(x,y_1,y_2) = E_1(x,y_1) \circ E_2(x,y_2)$$ is a strong $(k,\eps_1 + \eps_2 +
2^{-r})$-{\sf extractor}.
\end{theorem}


We can now prove Theorem~\ref{thm-better-extractor}.
Let $E$ be the $(k,\epsilon)$-extractor with seed $O(c_1 \log n)$ of Theorem~\ref{thm-extractor}.
By Theorem~\ref{thm-rsw}, we get an explicit strong $(k, \sqrt{\epsilon})$-extractor $E'$ with entropy loss $O(k \frac{\log\log n}{\log n})$.
We now iteratively apply Theorem~\ref{thm-wz} as follows. Let
$E^{(0)} = E'$. For each $1 < i \leq O(c_2)$, let $E^{(i)} : \{0,1\}^n \times
\{0,1\}^{d_{i}} \rightarrow \omm^{m_{i}}$ be the
strong $(k, \epsilon_{i})$-extractor produced by Theorem~\ref{thm-wz} when we take $E_1 = E^{(i-1)}$ and $E_2$ to be the strong $( k - m_{i-1} - c_1 \log n, 1/\log^{c_1} (n))$-extractor with seed length $O(c_1 \log n)$ given by Theorem~\ref{thm-extractor} and Theorem~\ref{thm-rsw}. Thus,
$$ d_{i} = O(i  c_1 \log n).$$
$$ \epsilon_{i} = O\left(\frac{i}{\log^{c_1}(n)}\right).$$
$$ m_{i} = m_{i-1} + (k - m_{i-1} - c_1 \log n)\left(1- O\left(\frac{\log\log n}{\log n}\right)\right).$$
Thus the entropy loss of $E^{(i)}$ is given by:
$$ k- m_{i} = (k - m_{i-1})\left(1 - \left(1- O\left(\frac{\log \log n}{\log n}\right)\right)\right) + O(c_1 \log n) = O\left(\frac{k}{\log^{i}(n)}\right).$$

$E^{(O(c_2))}$ is the desired extractor.
\qed

\begin{remark}
In fact \cite{GUV} and \cite{RRV} show how to extract {\emph all} the minentropy
with polylogarithmic seed length. Combined with the lossless condenser of \cite{GUV}
this gives an extractor that uses logarithmic seed to extract all the minentropy from sources
that have minetropy rate at most $2^{O(\sqrt{\log n})}$.

\begin{theorem}{\bf (Corollary of \cite[Theorem 4.21]{GUV})}\label{thm-guv-rrv}
For all positive integers $n \geq k$ such that $k = 2^{O(\sqrt{\log n})}$, and for all $\epsilon > 0$ there exists an
explicit $(k,\eps)$-extractor $E : \omm^n \times \omm^d \mapsto
\omm^m$ with $d = O(\log(n))$ and $m = k + d - 2 \log(1/\epsilon) - O(1)$.
\end{theorem}

This result combined with Theorem~\ref{thm-better-extractor} gives
an extractor with improved sub-linear entropy loss that works for
sources of all entropy rates, thus completing the proof of
Theorem~\ref{thm-final-extractor}.
\end{remark}


\section{Bounds on the list size for list-decoding Reed-Solomon codes}
\label{asec:rs}

In this section, we give a simple algebraic proof of an upper bound on the
list size for list-decoding Reed-Solomon codes within the Johnson radius.

Before stating and proving the theorem, we need some definitions.
For a bivariate polynomial $P(X,Y) \in \F[X,Y]$, we define its
$(a, b)$-degree to be the maximum of $ai + bj$ over all $(i,j)$
such that the monomial $X^iY^j$ appears in $P(X,Y)$ with a nonzero coefficient.
Let $N(k,d,\theta)$ be the number of monomials $X^iY^j$ which have
$(1,k)$-degree at most $d$ and $j \leq \theta d/k$. We have the following simple fact.
\begin{fact}
\label{fact1}
For any $k < d$ and $\theta \in [0, 1]$, $N(k, d, \theta) >  \theta\cdot(2-\theta)\cdot \frac{d^2}{2k}$.
\end{fact}

Now we prove the main theorem of this section. The proof is an enhancement
of the original analysis of the Guruswami-Sudan algorithm using the extended method
of multiplicities.

\begin{theorem}[List size bound for Reed-Solomon codes]
\label{thm:rsbound}
Let $(\alpha_1, \beta_1), \ldots, (\alpha_n, \beta_n) \in \F^2$.
Let $R, \gamma \in [0, 1]$ with $\gamma^2 > R$.
Let $k = Rn$. Let $f_1(X), \ldots, f_L(X) \in \F[X]$ be polynomials
of degree at most $k$, such that for each $j \in [L]$ we have
$|\{ i \in [n]: f_j(\alpha_i) = \beta_i\}| > \gamma n$.
Then $L \leq \frac{2 \gamma}{\gamma^2 - R}$.
\end{theorem}
\begin{proof}
Let $\epsilon > 0$ be a parameter. Let $\theta = \frac{2}{\left(1 + \frac{\gamma^2}{R}\right)}$.
Let $m$ be a large integer (to be chosen later), and let $d = \left(1 + \epsilon\right) \cdot m \cdot \sqrt{\frac{nk}{\theta \cdot (2-\theta)}}$.
We first interpolate a nonzero polynomial $P(X,Y)\in \F[X,Y]$ of $(1, k)$-degree at most $d$ and
$Y$-degree at most $\theta d/k$, that vanishes with multiplicity at least $m$ at each of the
points $(\alpha_i, \beta_i)$. Such a polynomial exists if $N(k, d, \theta)$, the number of monomials
available, is larger than the number of homogeneous linear constraints imposed by the vanishing conditions:
\begin{equation}
\frac{m (m+1)}{2} \cdot n  < N(k, d, \theta).
\end{equation}
This can be made to hold by picking $m$ sufficiently large, since by Fact~\ref{fact1},
$$N(k, d, \theta) > \theta \cdot (2-\theta) \frac{ d^2}{2k} = \frac{(1+\epsilon)^2 m^2}{2} \cdot n.$$

Having obtained the polynomial $P(X,Y)$, we also view it as a univariate polynomial
$Q(Y) \in \F(X)[Y]$ with coefficients in, $\F(X)$, the field of rational functions in $X$.

Now let $f(X)$ be any polynomial of degree at most $k$ such that, letting
$I = \{ i \in [n]: f(\alpha_i) = \beta_i\}$, $|I| \geq A$. We claim that
the polynomial $Q(Y)$ vanishes at $f(X)$ with multiplicity at least $m - d/A$.
Indeed, fix an integer $j < m-d/A $, and let $R_j(X) = Q^{(j)}(f(X))= P^{(0,j)}(X, f(X))$.
Notice the degree of $R_j(X)$ is at most $d$.
By Proposition~\ref{prop:multmult} and Lemma~\ref{lem:dermult},
$$\mult(R_j, \alpha_i) \geq \mult(P^{(0,j)}, (\alpha_i, \beta_i))
\geq \mult(P, (\alpha_i, \beta_i)) - j.$$
Thus
$$\sum_{i \in I} \mult(R_j, \alpha_i) \geq |I| \cdot (m-j) \geq
A \cdot (m-j) > d.$$
By Lemma~\ref{lem:schwartz}, we conclude that $R_j(X) = 0$. Since this holds for
every $j < m-d/A$, we conclude that $\mult( Q, f(X)) \geq m-d/A$.

We now complete the proof of the theorem. By the above discussion,
for each $j \in [L]$, we know that $\mult(Q, f_j(X)) \geq m-\frac{d}{\gamma n}$.
Thus, by Lemma~\ref{lem:schwartz} (applied to the nonzero polynomial $Q(Y) \in \F(X)[Y]$ and
the set of evaluation points $S = \{ f_j(X) : j \in [L] \}$)
$$\deg(Q) \geq \sum_{j \in [L]} \mult(Q, f(X)) \geq \left(m - \frac{d}{\gamma n}\right)\cdot L.$$
Since $\deg(Q) \leq \theta d/k$, we get,
$$ \theta d/k \geq \left(m - \frac{d}{\gamma n}\right)\cdot L.$$
Using $d = (1+\epsilon) \cdot m \cdot \sqrt{\frac{nk}{\theta \cdot (2-\theta)}}$ and $\theta = \frac{2}{1 + \frac{\gamma^2}{R}}$, we get,
$$ L \leq \frac{\theta}{k \cdot \frac{m}{d} - \frac{k}{\gamma n}} = \frac{\theta}{\frac{1}{1+\epsilon} \sqrt{\frac{k}{n}\cdot \theta \cdot (2-\theta)} - \frac{k}{ \gamma n}} = \frac{1}{\frac{1}{1+\epsilon} \sqrt{R\left(\frac{2}{\theta} - 1\right)} - \frac{R}{\theta \gamma} } = \frac{1}{\frac{\gamma}{1 + \epsilon} - \left(\frac{\gamma}{2}  + \frac{R}{2\gamma}\right)}.$$
Letting $\epsilon \rightarrow 0$, we get $L \leq \frac{2\gamma}{\gamma^2 - R}$, as desired.
\end{proof}



\newpage
%


\end{document}